\documentclass[12pt]{amsart}

\numberwithin{equation}{section}

\usepackage{amsmath,amsthm,amsfonts,amscd,eucal}
\usepackage{url}
\usepackage{graphicx}
\usepackage{setspace}
\usepackage{hyperref}
\usepackage{amssymb}
\usepackage[british]{babel}
\usepackage{dsfont}

\usepackage{tikz}
\usepackage{xcolor}
\usepackage{pgfplots}
\usepackage{float}
\usetikzlibrary{plotmarks}
\usetikzlibrary {arrows.meta}

\pgfplotsset{compat = newest}
 
\pgfplotsset{every axis/.append style={
    font=\large,
    line width=1pt,
    tick style={line width=0.8pt}}}

\def\ch{{\mathcal H}}

\def\Om{\Omega}
\def\S{\Sigma}
\def\a{\alpha}
\def\b{\beta}
\def\veps{\varepsilon}

\def\gf{{\mathfrak F}}

\newtheorem{thm}{Theorem}[section]
\newtheorem{lem}[thm]{Lemma}

\newtheorem{prop}[thm]{Proposition}

\theoremstyle{definition}
\newtheorem{rem}[thm]{Remark}
\newtheorem{defin}[thm]{Definition}

\def\dim{\mathop{\rm dim}}
\def\spn{\mathop{\rm span}}
\def\aut{\mathop{\rm Aut}}
\def\id{\mathop{\mathds{1}}}

\newcommand{\D}{\mathfrak{D}}
\newcommand{\DN}{\Delta(n)}
\newcommand{\DD}{\Delta}
\newcommand{\M}{\mathcal{M}}
\newcommand{\N}{\mathbb{N}}
\newcommand{\vp}{\varphi}
\newcommand{\ve}{\varepsilon}
\newcommand{\R}{\mathbb{R}}

\begin{document}

\title{Finitely Generated Weakly Monotone $C^*$-algebra}
\author{Maria Elena Griseta}
\address{Maria Elena Griseta\\
Dipartimento di Matematica\\
Universit\`{a} degli studi di Bari\\
Via E. Orabona, 4, 70125 Bari, Italy}
\email{\texttt{mariaelena.griseta@uniba.it}}
\author{Janusz Wysocza\'nski}
\address{Janusz Wysocza\'nski\\
 Mathematical Institute, Wroclaw University, pl.  Grunwaldzki 2, 50–384 Wroclaw, Poland}
\email{\texttt{janusz.wysoczanski@math.uni.wroc.pl}}

\date{\today}

\maketitle

\section*{abstract}
\label{abstract}
We consider the $C^*$-algebra generated by finitely many annihilation operators acting on the weakly monotone Fock space, and we call it weakly monotone $C^*$-algebra. We give an abstract presentation for this algebra, showing that it is isomorphic to a suitable quotient of a Cuntz-Krieger $C^*$-algebra $\mathcal{O}_A$ corresponding to a suitable matrix $A$.
Furthermore, we show that the diagonal subalgebra of the weakly monotone $C^*$-algebra is a MASA and we give the detailed description of its Gelfand spectrum.

\bigskip

\vskip0.1cm\noindent \\
{\bf Mathematics Subject Classification}: 46L05, 46K10.  \\
{\bf Key words}: weakly monotone $C^*$-algebras, Cuntz-Krieger $C^*$-algebras, maximal abelian subalgebra.\\

\section{introduction}
\label{intro}
The celebrated papers by Cuntz and Krieger \cite{cuntzKri80} and Cuntz \cite{Cuntz81} studied universal $C^*$-algebras generated by nonzero partial isometries $S_i$, indexed by a \textit{finite} set $\Sigma$, with relations, known as  $(CK1)$ and $(CK2)$,  for their support projections $Q_i=S_i^*S_i$ and range projections $P_i=S_iS_i^*$, later called the Cuntz-Krieger relations, given by a finite matrix $\displaystyle A=(a_{ij})_{i,j\in \Sigma}$, with entries $a_{ij}\in \{0, 1\}$. 
\begin{eqnarray*}
(CK1)&& P_iP_j = 0\quad \text{if} \quad i\neq j \\
(CK2)&& Q_i=\sum_{j\in\Sigma} a_{ij}P_j
\end{eqnarray*} 
In particular, the range projections were assumed $(CK1)$ to be mutually orthogonal and the support projections were finite sums of some of the range projections, chosen by the given matrix $A$. It was observed in \cite{KPRR1997}, and then developed in \cite{KPR1998}, that the matrix $A$ can be thought of as an incidence matrix of a directed graph, in which the range projections are associated with vertices, and the partial isometries are associated with directed edges. This approach has led to the development of the study of \textit{graph} $C^*$\textit{-algebras}. 
In this context Hong and Szymański studied in \cite{HSz2002} $C^*$-algebras related to odd-dimensional quantum spheres $C^*(S_q^{2n-1})$, which are finitely generated by elements which satisfy Twisted Canonical Commutation Relations (TCCR) with the parameter $q\in [0,1)$, defined by Pusz and Woronowicz in \cite{PW1989}. The construction \cite{PW1989} is based on q-deformed twist operator $T_q$, defined on the orthonormal basis $\{e_j:1\leq j \leq n\}$ of the tensor product $H\otimes H$ of a Hilbert space  $H$ by the formula 
\begin{eqnarray*}
T_{q}(e_j\otimes e_k) = 
\begin{cases}
q (e_k\otimes e_j) & \text{if} \quad j<k \\ 
q^2 (e_j\otimes e_j) & \text{if} \quad j=k \\ 
(q^2-1) (e_j\otimes e_k)+q (e_k\otimes e_j) & \text{if} \quad j>k 
\end{cases}   
\end{eqnarray*}
The commutation relations TCCR, described in \cite{PW1989}, which follow, are of the form (\cite{HSz2002}, Section 4)
\begin{eqnarray*}
1&=& z_1z_1^*+\cdots + z_nz_n^* \\
z_iz_j&=& qz_jz_i \quad \text{if} \quad i<j \\ 
z_j^*z_i&=&  qz_iz_j^* \quad \text{if} \quad i\neq j \\ 
z_i^*z_i&=& z_iz_i^* + (1-q^2) \sum_{j<i} z_jz_j^* \quad \text{for} \quad 1\leq i \leq n  
\end{eqnarray*}
and are studied there as generating universal $C^*$-algebra of the odd-dimensional quantum sphere $S_q^{2n-1}$. In particular, an abelian subalgebra, generated by $z_jz_j^*$ is considered in \cite{HSz2002} and its spectrum is described. As shown by Bożejko \cite{B2012}, the deformed twist operator $T_q$ defines deformation of the full Fock space, on which creation and annihilation operators are defined. In the particular case $q=0$ the $T_0$ twist operator leads to the  \textit{weakly monotone Fock space}, originally defined by the second-named author in \cite{JWys2005}, and subsequently studied in \cite{CGW20} and \cite{CGW21}. The weakly monotone Fock space served as a model for monotone independence of Muraki \cite{Mur2001}, and was generalized in \cite{JWys07, JWys10} to algebras indexed by arbitrary partially ordered set, instead of the totally ordered set of positive integers, still the creation and annihilation operators there are partial isometries.
 
In the present paper we study the $C^*$-algebra $C^*(A_1, \ldots , A_n)$  generated by the annihilation operators $A_i$, (as well as the creation operators $A^{\dag}_i$), for $1\leq i \leq n$, which act on the weakly monotone Fock space. They correspond to the notation of \cite{HSz2002} as 
\[
z_i^*:=A_i, \quad z_i:=A^{\dag}_i \quad \text{for} \quad q=0. 
\]
This algebra contains the vacuum projection $P_{\Omega}=A_0:=A_1A^{\dag}_1-A^{\dag}_1A_1$, whereas the corresponding relation in \cite{HSz2002} is $z_1^*z_1-z_1z_1^*=0$, so that $z_1$ is normal. Moreover, it is unital as the (finite) sum of the projections $A_0+A^{\dag}_1A_1+\cdots + A^{\dag}_nA_n$ is its unit.  In addition, from the point of view of (CK1)-(CK2) relations and graph algebras, the defining matrix $A=(a_{ij})_{i,j=0}^{n}$ is lower-triangular with entries $1$ on main diagonal and below, and $0$ above it.   In this framework, we show that the weakly monotone $C^*$-algebra is a quotient of the Cuntz-Krieger algebra $\mathcal{O}_A$ associated with this matrix. The case concerning the $C^*$-algebras generated by an infinite family of creation and annihilation operators $A_i$, $i\in \mathbb{N}$ and $i\in \mathbb{Z}$, respectively, has been studied in \cite{CDRW}, where the relations with the Exel-Laca algebras \cite{exelLaca99} have been investigated.  

One of our main results is the description of the \textit{maximal abelian subalgebra} $\D_n\subset C^*(A_1, \ldots , A_n)$ and its spectrum, i.e. the  Gelfand space. In particular, we show that the spectrum can be identified as a compact subset of the $n$-dimensional cube $[0,\ 1]^n$, where the accumulation points lie on the edges. Pictures of the cases $n=2$ and $n=3$ are provided for the reader's convenience, to demonstrate the crucial properties of the spectrum. 

The paper is organized as follows. In Section \ref{sec:prel}, we recall some definitions and notions concerning the weakly monotone Fock space, the Cuntz-Krieger algebras and the $C^*$-algebras generated by isometries. In Section \ref{sec:WMalgebra}, we study the $C^*$-algebra generated by $n\geq2$ annihilation operators acting on this space, showing that it is a faithful representation for a suitable quotient of $\mathcal{O}_A$. Section \ref{sec:MASA} is devoted to characterizing a MASA for the weakly monotone $C^*$-algebra. More in detail, Section \ref{subsec:MASA2} deals with the case of $n=2$ generators and in Section \ref{subsec:MASAn} we give analogous considerations to determine a MASA for the $C^*$-algebra generated by $n$ annihilation operators. Finally, in Section \ref{subsec:spectrum} we study the spectrum of the maximal abelian subalgebra.

\section{preliminaries}
\label{sec:prel}
This section gives a miscellany of definitions, notations, and some known results often recalled in the sequel. \\
We start by recalling the definition of weakly monotone Fock space and creation and annihilation operators acting on it. For a fuller treatment, we refer the reader to \cite{JWys2005}.

Let $\ch$ be a separable Hilbert space with a fixed orthonormal basis $\{e_i:i\geq1\}$. By $\gf(\ch)$ we denote the full Fock space on $\ch$, whose vacuum vector is $\Om=1\oplus0\oplus\cdots$. The \emph{weakly monotone Fock space}, in the sequel denoted by $\gf_{WM}(\ch)$, is the closed subspace of $\gf(\ch)$ spanned by $\Om$, $\ch$ and all the simple tensors of the form $e_{i_k}\otimes e_{i_{k-1}}\otimes \cdots\otimes e_{i_1}$, where $i_k\geq i_{k-1}\geq\cdots\geq i_1$ and $k\geq2$.

If the Hilbert space $\ch$ is finite dimensional with $n =\dim(\ch)\geq 2$, then the basis $\mathbf{B}$ for $\gf_{WM}(\ch)$ consists of the vacuum and all the simple tensors
\begin{equation}
\label{basis}
e_n^{k_n}\otimes e_{n-1}^{k_{n-1}}\otimes\cdots\otimes e_1^{k_1}\,,
\end{equation}
where $k_n, k_{n-1},\ldots,k_1\geq 0$, $e_m^{k}:=\underbrace{e_m\otimes\cdots\otimes e_m}_k$ if $k\geq1$, and the convention that $e^{k_i}_{i}$ does not appear in \eqref{basis} if $k_i=0$.

The weakly monotone creation and annihilation operators with "test function" $e_i$, denoted by $A^\dag_i$ and $A_i$ respectively, are defined on the linear generators as follows. For any $i_k\geq i_{k-1}\geq \cdots \geq i_1$, $k\geq 2$, and $j\geq 1$
\begin{align*}
&A_i\Om=0\,, \quad A_i(e_j)=\delta_{ij}\Om\,, \\
& A_i(e_{i_k}\otimes e_{i_{k-1}}\otimes \cdots\otimes e_{i_1})=\delta_{ii_k}e_{i_{k-1}}\otimes \cdots\otimes e_{i_1}\,,
\end{align*}
where $\delta_{ij}$ is the Kronecker symbol, and
\begin{align*}
&A^{\dag}_i(\Om)=e_i\,, \quad A^{\dag}_ie_j=\a_{ij}e_i\otimes e_j\,, \\
&A^{\dag}_i(e_{i_k}\otimes e_{i_{k-1}}\otimes \cdots\otimes e_{i_1})=\a_{ii_k} e_i\otimes e_{i_k}\otimes e_{i_{k-1}}\otimes \cdots\otimes e_{i_1}\,,
\end{align*}
where
\begin{equation*}
\a_{j,k}=\begin{cases}
1 & \text{if $j\geq k$,}\\
0 & \text{otherwise.}
\end{cases}
\end{equation*}
They can be extended by linearity and continuity to the whole $\gf_{WM}(\ch)$, where they are adjoint to each other, and with unit norm. Furthermore, they satisfy the following relations
\begin{equation}
\begin{split}
\label{cr}
A^{\dag}_iA^{\dag}_j&=A_jA_i =0 \quad\text{if $i<j$,}\\
A_iA^{\dag}_j& =0 \,\,\,\,\,\,\,\,\,\,\,\,\,\,\,\,\,\,\,\,\,\,\,\,\,\,\, \text{if $i\neq j$.} \\
\end{split}
\end{equation}

In what follows we recall the notion of Cuntz-Krieger algebras, and we refer the reader to \cite{cuntzKri80} for a deeper discussion.
Let $\Sigma$ a finite set, and $A=(a_{ij})_{i,j\in\Sigma}$ a complex matrix, with $a_{ij}\in\{0,1\}$ for each $i,j\in\Sigma$. In addition, suppose that in each row or column of $A$ there exists at least one element different from zero. Let $S_i$, $i\in\Sigma$ the non-zero partial isometries acting on the Hilbert space such that, for each $i\in\Sigma$, their support projections $Q_i:=S^{\ast}_iS_i$ and their range projections $P_i:=S_iS^{\ast}_i$ satisfy the following relations:
\begin{description}
  \item[CK1)] $P_iP_j=0$ for each $i\neq j$;
  \item[CK2)] $\displaystyle Q_i=\sum_{j\in\Sigma} a_{ij}P_j$, $i,j\in\Sigma$.
\end{description}
\begin{rem}
\label{rem:CK1}
We notice that the first condition $CK1)$ is equivalent to the following condition given in \cite{exelLaca99}, which we again denote by $CK1)$.
\begin{description}
  \item[CK1)] $\displaystyle I=\sum_{j\in\Sigma}S_jS^{\ast}_j$,
\end{description}
where $I$ denotes the identity operator on the Hilbert space $\ch$.
\end{rem}
Denote by $\mathcal{O}_A$ the $C^*$-algebra generated by the non-zero partial isometries $S_i$, $i\in\Sigma$, that satisfy the Cuntz-Krieger relations $CK1)$ and $CK2)$. Let $\mu=(i_1,\ldots,i_k)$, $i_j\in\Sigma$, a multiindex and denote by $S_{\emptyset}=I$, $S_{\mu}=S_{i_1}S_{i_2}\cdots S_{i_k}$. For each $\mu$, $S_{\mu}$ is a partial isometry and $S_{\mu}\neq 0$ if and only if $A(i_j,i_{j+1})=1$, ($j=1,\ldots, k-1$). Denote by $\mathcal{M}_A$ the set of all multi-indices $\mu$ such that $S_{\mu}\neq0$ and by $\Sigma_0$ the set of all $i\in\Sigma$ for which there are at least two different multi-indices $\mu=(i_1,\ldots,i_r)$ and $\nu=(j_1,\ldots,j_s)$ belonging to $\mathcal{M}_A$, such that $i_1=i_r=j_1=j_s=i$ ($r,s\geq2$),  while $i_k$, $j_l\neq i$ for each $1 <k<r$, $1 <l<s$.  Suppose that $A$ satisfies the following condition
\begin{description}
  \item[CK3)] For each $i\in\Sigma$ there exists $\mu=(i_1,\ldots,i_r)$, ($r\geq1$) belonging to $\mathcal{M}_A$ such that $i_1=i$ and $i_r\in\Sigma_0$.
\end{description}
If the matrix $A$ satisfy the condition $CK3)$, one says that the Cuntz-Krieger $C^*$-algebra $\mathcal{O}_A$ has the universal property (see \cite[Theorem 2.13]{cuntzKri80}). \\
The Cuntz-Krieger algebras $\mathcal{O}_A$ extend the Cuntz algebras $\mathcal{O}_n$, in which the $n\times n$ matrices $A$ are given by $a_{ij}=1$  for each $i,j\in\{1,\ldots,n\}$ (see \cite{cuntz77}).\\
Concerning the case in which the matrix $A$ does not satisfy the condition $CK3)$, in \cite{huef} the authors introduce a universal Cuntz-Krieger algebra $\mathcal{A}\mathcal{O}_A$ that coincides with $\mathcal{O}_A$ if condition $CK3)$ is satisfied. 

We next recall some important facts concerning the $C^*$-algebra generated by an isometry and we refer the reader to \cite{coburn67,davidson} for more details. We start by introducing several notations. Let us consider $\ch:=\ell_2(\mathbb{N})$ the Hilbert space with orthonormal basis $\{e_1,\ldots,e_n:n\geq1\}$ and denote by $\mathcal{B(\ch)}$ and $\mathcal{K}(\ch)$ the algebra of all bounded operators and linear compact operators on $\ch$, respectively. The unilateral shift $S$ is defined as $Se_n:=e_{n+1}$.\\

Here we take into account the general case in which $A\in \mathcal{B}(\ch)$ is an arbitrary isometry, i.e. $A^{\ast}A=I$. Exploiting the Wold's decomposition, we can write $A=W\oplus S_{\gamma}$, where $W$ is unitary and $S_{\gamma}$ denote the shift of multiplicity $\gamma$, i.e. the $\gamma$-fold direct sum $S\oplus S\oplus\cdots\oplus S$ acting on $\ch\oplus\ch\oplus\cdots\oplus\ch$. Then there is an isometric $*$-isomorphism between the $C^*$-algebra generated by the isometry $A$, $C^*(A)$, and the $C^*$-algebra $C^*(W\oplus S)$. Therefore $C^*(A)\cong C^*(S)$ and there is a unique minimal ideal $\mathcal{I}(A)\neq0$ such that $C^*(A)/\mathcal{I}\cong \mathcal{C}(\mathbb{T})$, the algebra of all continuous functions on the perimeter of the unit circle $\mathbb{T}:=\{z\in\mathbb{C}:|z|=1\}$.
If the isometry $A$ is unitary, then $C^*(A)$ is isometrically $*$-isomorphic to $\mathcal{C}(\sigma(A))$, the algebra of all continuous functions defined onto the spectrum of $A$. \\
In particular, by choosing $\ch=\mathbb{C}$ and the isometry $A:=\lambda \id$, $\lambda\in\mathbb{C}$, we obtain the one dimensional irreducible representations for the $C^*$-algebra generated by the isometry $A$, $\pi_{\lambda}(A):=\lambda I$, for each $\lambda\in\mathbb{C}$.

\section{The weakly monotone $C^*$-algebra}
\label{sec:WMalgebra}

In this section, we define the weakly monotone $C^*$-algebra generated by $n\geq 2$ annihilation operators acting on the weakly monotone Fock space, and we prove that it is isomorphic to a suitable quotient of the Cuntz-Krieger algebra $\mathcal{O}_A$, for a given matrix $A$.  \\

\begin{rem}
\label{rem:isometry}
We notice that for each $i=1,2,\ldots,n$, the weakly monotone creation operators $A^{\dag}_i$ are partial isometries, indeed $A_iA^{\dag}_i$ are projections onto the initial spaces $\gf_{WM}^{\leq i}$ for the operators $A^{\dag}_i$, given by
\begin{equation*}
    \gf_{WM}^{\leq i}(\ch):=\overline{\spn\{\Om,e_{k_m}\otimes \cdots \otimes e_{k_1}: i\geq k_m\geq \ldots \geq k_1, \ m\geq 1\}}\,.
\end{equation*}
In addition, the creator $A^{\dag}_n$ is an isometry, indeed $(A^{\dag}_n)^{\ast}A^{\dag}_n=A_nA^{\dag}_n=I$, where $I$ denotes the identity operator on $\gf_{WM}(\ch)$.
\end{rem}
\begin{rem}
\label{rem:M1M2}
Let $n=\dim(\ch)$. By \eqref{cr}, it follows that:
\begin{description}
  \item[M1)] $A^{\dag}_{n-1}A^{\dag}_n=0$ if and only if $A_nA_{n-1}=0$;
  \item[M2)] $A_{n-1}A^{\dag}_n=0$.
\end{description}
\end{rem}

Denote by $C^*(A_1,\ldots, A_n)$ the concrete $C^*$-algebra generated by the annihilation operators $A_1,\ldots, A_n$ acting on $\gf_{WM}(\ch)$. The goal of this section is to give an abstract representation for this weakly monotone $C^*$-algebra. To this end, we will show that it is isomorphic to a suitable quotient of the Cuntz-Krieger algebras described in Section \ref{sec:prel}. \\
We start by proving the following Lemma, which provides several relations between the weakly monotone creation and annihilation operators and $A_0:=P_{\Om}=A_1A^{\dag}_1-A^{\dag}_1A_1$, the orthogonal projection of the weakly monotone Fock space $\gf_{WM}(\ch)$ onto the subspace $\mathbb{C}\Om$. We recall that $A^{\ast}_i=A^{\dag}_i$ for each $i=1,\ldots,n$. \\
\begin{lem}
\label{lem:123}
The following relations are satisfied:
\begin{description}
  \item[1)] $A_0$ is normal, i.e. $A_0A^{\ast}_0=A^{\ast}_0A_0$;
  \item[2)] $\displaystyle A_iA^{\ast}_i=\sum_{j=0}^iA^{\ast}_jA_j$, for each $i=1,\ldots,n$;
  \item[3)] $A_nA^{\ast}_n=\displaystyle\sum_{j=0}^nA^{\ast}_jA_j=I$.
\end{description}
\end{lem}
\begin{proof}
The first sentence follows immediately after observing that $A_0$ is self-adjoint, since it is an orthogonal projection. Concerning the second relation, for each $i=1,\ldots,n$, one has
\begin{equation*}
    A_iA^{\ast}_i(\Om)=A_i(e_i)=\Om \quad\text{and}\quad \sum_{j=0}^iA^{\ast}_j A_j(\Om) = A^{\ast}_0A_0(\Om)=\Om\,.
\end{equation*}
In addition, for each $k_l, k_{l-1},\ldots,k_1\geq 0$, with $\displaystyle k_l>0$, one has
\begin{equation*}
     A_iA^{\ast}_i(e_l^{k_l}\otimes e_{l-1}^{k_{l-1}}\otimes\cdots\otimes e_1^{k_1}) = \left\{
       \begin{array}{ll}
         0, & \hbox{\text{if $i<l$};} \\
         e_l^{k_l}\otimes e_{l-1}^{k_{l-1}}\otimes\cdots\otimes e_1^{k_1}, & \hbox{\text{if $i\geq l$.}}
       \end{array}
     \right.
\end{equation*}
On the other hand,
\begin{equation*}
\begin{split}
&\sum_{j=0}^iA^{\ast}_j A_j(e_l^{k_l}\otimes e_{l-1}^{k_{l-1}}\otimes\cdots\otimes e_1^{k_1}) \\
&=\left\{
       \begin{array}{ll}
         0, & \hbox{\text{if $i<l$};} \\
         A^{\ast}_l A_l (e_l^{k_l}\otimes e_{l-1}^{k_{l-1}}\otimes\cdots\otimes e_1^{k_1}), & \hbox{\text{if $i\geq l$.}}
       \end{array}
     \right.\\
     & =\left\{
       \begin{array}{ll}
         0, & \hbox{\text{if $i<l$};} \\
        e_l^{k_l}\otimes e_{l-1}^{k_{l-1}}\otimes\cdots\otimes e_1^{k_1}, & \hbox{\text{if $i\geq l$.}}
       \end{array}
     \right.
\end{split}
\end{equation*}
Concerning the last equality, from Remark \ref{rem:isometry}, it follows that $\displaystyle A_nA^{\ast}_n=I$.
\end{proof}
Observe that the non-zero partial isometries $S_i$ described in Section \ref{sec:prel} correspond to the operators $A^{\ast}_i$, $i\in\Sigma:=\{0,1,\ldots,n\}$, described in this section. In the following lines, we show that, for each $i\in\Sigma$, the operators $A^{\ast}_i$ satisfy the Cuntz-Krieger relations $CK1)$ and $CK2)$ for a suitable matrix $A$. \\
Therefore, let us consider $P_i:=A^{\ast}_iA_i$ and $Q_i:=A_iA^{\ast}_i$, for each $i\in\Sigma$. By \eqref{cr} it follows that $P_iP_j=0$, for each $i\neq j$. In addition, defining $A=(a_{ij})_{i,j\in\Sigma}\in \mathrm{M}_{n+1}(\mathbb{C})$ as
\begin{equation}
\label{matrixAn}
A:=
\begin{pmatrix}
1 & 0 & 0 & 0 & \cdots &0\\
1 & 1 & 0 & 0 & \cdots &0\\
  & \vdots & & &  \vdots \\
1 & 1 & 1 & 1 & \cdots & 1\\
\end{pmatrix},
\end{equation}
we observe that the conditions $CK2)$ are also satisfied for the given matrix $A$. More in detail, by Lemma \ref{lem:123}, for each $i=1,\ldots,n$, one has
  \begin{align*}
    Q_0 &=A_0A^{\ast}_0=A^{\ast}_0A_0=P_0=\sum_{j\in\Sigma} a_{0j}P_j\\
    Q_i &=A_iA^{\ast}_i =\sum_{j=0}^iA^{\ast}_jA_j= \sum_{j=0}^i P_j=  \sum_{j\in\Sigma} a_{ij}P_j
  \end{align*}
Notice that the $C^*$-algebra generated by the partial isometries $A_i$, $i\in\S$, coincide with $C^*(A_1,\ldots,A_n)$, since the orthogonal projection $A_0$ can be obtained by $A_i$ and $A^{\ast}_i$, by the condition $CK2)$.\\
Given the matrix $A$ defined in \eqref{matrixAn}, we can consider the abstract $C^*$-algebra $\mathcal{O}_A$, generated by $n+1$ non-zero partial isometries $a_i$, $i\in\Sigma$, which satisfy the conditions $CK1)$ and $CK2)$. The matrix $A$ does not satisfy the condition $CK3)$, then $\mathcal{O}_A$ has not the universal property. In addition, from $CK2)$, it follows that $a_0$ is normal, but we cannot assert that $a_0$ is positive. On the contrary, the orthogonal projection $A_0$ is positive. Therefore, let us consider the ideal of the algebra $\mathcal{O}_A$ generated by $a_0-a^{\ast}_0a_0$, $\mathcal{I}:=<a_0-a^{\ast}_0a_0>\trianglelefteq  \mathcal{O}_A$. The main result of this section consists of showing that the $C^*$-algebra $C^*(A_1,\ldots,A_n)$ is isomorphic to the quotient  $\mathcal{O}_A/\mathcal{I} $.\\
We start by making several remarks and by proving some preparatory results.
\begin{rem}
\label{rem:M1M2OA}
The relations $M1)$ and $M2)$ given in Remark \ref{rem:M1M2} still holds on the $C^*$-algebra $\mathcal{O}_A$ and on the quotient algebra $\mathcal{O}_A/\mathcal{I} $.
\end{rem}
Here we list some irreducible representations for the $C^*$-algebra $\mathcal{O}_A$.
\begin{itemize}
  \item Fix $z\in\mathbb{T}$ and denote by $\ch_{z}:=\mathbb{C}$ the complex space.  The one-dimensional irreducible representations $(\ch_{z},\pi_{z})$ on the $C^*$-algebra generated by $A_1,\ldots,A_n$, are given by
      \begin{align*}
       \pi_{z}(A_i)&=0 \quad \text{for each $i=0,\ldots,n-1$}   & \pi_{z}(A_n)&=z I\,.
      \end{align*}
  \item Fix $z\in\mathbb{T}$ and denote by $\ch_{z}:=\gf_{WM}(\ch)$ the weakly monotone Fock space. The Fock representations $(\ch_{z}, \pi^{z}_F)$ on the abstract Cuntz-Krieger $C^*$-algebra generated by the non-zero partial isometries $a_i$, $i\in\S$, are given by
      \begin{align}
      \label{fockRep}
              \pi^{z}_F(a_0)&=z P_{\Om}=z A_0 & \pi^{z}_F(a_i)&=A_i \quad \text{for each $i\in\{1,\ldots,n\}$}\,.
      \end{align}
\end{itemize}
Our next goal is to explicitly exhibit a
faithful representation for the $C^*$-algebra $\mathcal{O}_A$.\\
Let $A\in M_{n+1}(\mathbb{C})$ the matrix given in \eqref{matrixAn} and $\mathcal{O}_A$ the $C^*$-algebra generated by the partial isometries $a_i$, $i\in\S$, satisfying the conditions $CK1)$ and $CK2)$. For each $z\in\mathbb{T}$, denote by $\a_z\in\aut(\mathcal{O}_A)$ the natural gauge action of the torus $\mathbb{T}$ on $\mathcal{O}_A$, defined as follows:
\begin{equation*}
    \alpha_z(a_i):=za_i\quad \forall z\in\mathbb{T}\,, i\in\S\,.
\end{equation*}
We notice that for each $i\in\S$, the operators $a^{\prime}_i:=\alpha_z(a_i)=za_i$ are partial isometries and satisfy the Cuntz-Krieger relations $CK1)$ and $CK2)$. In addition, $\alpha_z\circ \alpha_w=\alpha_{zw}$ for each $z,w\in\mathbb{C}$.

\begin{rem}
\label{rem:gaugeFaith}
Let $\pi:\mathcal{O}_A\rightarrow \mathcal{B}(\ch)$ a representation of the $C^*$-algebra $\mathcal{O_A}$. Suppose that for every $z\in\mathbb{T}$ there exists an automorphism $\beta_z\in\aut(\pi(\mathcal{O_A}))$ such that $\beta_z(\pi(a_i))=z\pi(a_i)$, for each $i\in\S$. Then, by Theorem \cite[Theorem 2.3]{huef}, it follows that the representation $(\ch,\pi)$ is faithful.
\end{rem}

\begin{prop}
\label{prop:faithful}
Let $z\in\mathbb{T}$, $\ch_{z}:=\gf_{WM}(\ch)$ the weakly monotone Fock space, and $\pi^{z}_F$ the Fock representation given in \eqref{fockRep}. Then
$$
\bigg(\bigoplus_{z\in\mathbb{T}}\ch_{z},\bigoplus_{z\in\mathbb{T}}\pi^{z}_F\bigg)
$$
is a faithful representation for the $C^*$-algebra $\mathcal{O}_A$.
\end{prop}
\begin{proof}
Denote by $\displaystyle\tilde{\pi}:=\bigoplus_{z\in\mathbb{T}}\pi^{z}_F$ the representation of the $C^*$-algebra $\mathcal{O}_A$ and we will show that, for each $w\in\mathbb{T}$, there exists a gauge automorphism $\gamma_w\in\aut(\tilde{\pi}(\mathcal{O}_A))$, such that $\gamma_w(\tilde{\pi}(a_i))=w\tilde{\pi}(a_i)$, for each $i\in\S$. Then the thesis follows by Remark \ref{rem:gaugeFaith}. \\
Let $\displaystyle \widetilde{\ch}:= \bigoplus_{z\in\mathbb{T}}\ch_{z}$ and denote by $\mathcal{U}(\widetilde{\ch})$ the set of all unitary operators on $\widetilde{\ch}$. We will prove that, for each $w\in\mathbb{T}$, there exists a unitary operator $U_w\in \mathcal{U}(\widetilde{\ch})$ such that, for each $i\in\S$, one has
\begin{equation*}
    U_w\tilde{\pi}(a_i)U_w^{\ast}=w\tilde{\pi}(a_i)\,.
\end{equation*}
For each $i\in\S$, denote by $\tilde{\beta}_i:=\tilde{\pi}(a_i)$. By \eqref{fockRep}, it follows that
\begin{align}
\label{betatilde}
  \tilde{\beta}_0&=\bigoplus_{z\in\mathbb{T}} z A_0 & &\text{and} &  \tilde{\beta}_i & =\bigoplus_{z\in\mathbb{T}} A_i\quad i=1,\ldots,n
\end{align}
For each $w\in \mathbb{T}$, we start by defining the unitary operator $U_w$, on the element of the basis of a single weakly monotone Fock space $\ch_{z}$, $z\in \mathbb{T}$, as follows:
\begin{align*}
U_w(e_n^{k_n}\otimes e_{n-1}^{k_{n-1}}\otimes\cdots\otimes e_1^{k_1}) &:=\bar{w}^{\sum_{i=1}^nk_i}e_n^{k_n}\otimes e_{n-1}^{k_{n-1}}\otimes\cdots\otimes e_1^{k_1}\,,\\
U_w(\Om_{z})&:=\Om_{\bar{w}z}\,,
\end{align*}
where $k_n, k_{n-1},\ldots,k_1\geq 0$, $\displaystyle \sum_{j=1}^n k_j>0$, $\Om_{z}$ denotes the vacuum vector on $\ch_{z}$ and $\bar{w}$ the complex conjugate of $w\in\mathbb{T}$. \\
It is easy to check that for each $w\in \mathbb{T}$, $U^{\ast}_w=U_{\bar{w}}$.\\
Fix $z\in\mathbb{T}$, denote by $\beta_i:=\pi^{z}_F(a_i)$, for each $i\in\S$. More in detail,
\begin{align*}
    \beta_0 &=z A_0 & &\text{and} &   \beta_i &= A_i\,,\quad i=1,\ldots,n\,.
\end{align*}
Fix $w\in \mathbb{T}$, we prove that $U_w\beta_i U^{\ast}_w = w\beta_i$, for each $i\in \S$. Indeed,
\begin{equation*}
    \begin{split}
       U_w\beta_0 U^{\ast}_w (\Om_{z}) &= U_w\beta_0 (\Om_{wz}) = U_w(wz \Om_{wz}) \\
                                            & = w z \Om_{\bar{w}wz}=w z \Om_{z}\\
                                            & =wz A_0(\Om_{z})=w \beta_0(\Om_{z})\,, \quad z \in \mathbb{T}\,.
                                            \end{split}
\end{equation*}
In addition, for each $i=1,\ldots,n$,
\begin{equation*}
    \begin{split}
       U_w\beta_i U^{\ast}_w (e_i^{k_i}\otimes e_{i-1}^{k_{i-1}}\otimes\cdots\otimes e_1^{k_1}) &= U_w\beta_i (w^{\sum_{j=1}^ik_j}e_i^{k_i}\otimes e_{i-1}^{k_{i-1}}\otimes\cdots\otimes e_1^{k_1}) \\
                                                                                                &= U_w(w^{\sum_{j=1}^ik_j}e_i^{k_i-1}\otimes e_{i-1}^{k_{i-1}}\otimes\cdots\otimes e_1^{k_1}) \\
                                                                                                &= w^{\sum_{j=1}^ik_j}\bar{w}^{\sum_{j=1}^{i}k_j-1}e_i^{k_i-1}\otimes e_{i-1}^{k_{i-1}}\otimes\cdots\otimes e_1^{k_1}\\
                                                                                                &= w^{\sum_{j=1}^ik_j}\bar{w}^{\sum_{j=1}^ik_j}\bar{w}^{-1}e_i^{k_i-1}\otimes e_{i-1}^{k_{i-1}}\otimes\cdots\otimes e_1^{k_1}\\
                                                                                                & = we_i^{k_i-1}\otimes e_{i-1}^{k_{i-1}}\otimes\cdots\otimes e_1^{k_1}\\
                                                                                                &= w A_i(e_i^{k_i}\otimes e_{i-1}^{k_{i-1}}\otimes\cdots\otimes e_1^{k_1})\\
                                                                                                &= w \beta_i(e_i^{k_i}\otimes e_{i-1}^{k_{i-1}}\otimes\cdots\otimes e_1^{k_1})\,.
 \end{split}
\end{equation*}
The assertion follows by extending the unitary operators $U_w$ on $\tilde{\ch}$.
\end{proof}
\begin{rem}
\label{rem:spectrum}
After denoting by $\sigma(a_0)$, the spectrum of the operator $a_0$, one has
$$
\sigma(a_0)=\mathbb{T}\cup \{0\}\,.
$$
Indeed, by Proposition \ref{prop:faithful}, it follows that $\sigma(a_0)=\sigma(\tilde{\beta}_0)$, where $\tilde{\beta}_0$ is given by \eqref{betatilde}. In addition, $\tilde{\beta}_0$ is a unitary operator on $\widetilde{\ch}$ and, after denoting by $\mathfrak{g}:=\spn\{\Om_{z}:z\in \mathbb{T}\}^{\bot}$, one has
\begin{align*}
\tilde{\beta}_0\Om_{z} &= \bigoplus_{z\in\mathbb{T}}z \Om_{z} \quad z\in\mathbb{T}\\
\tilde{\beta}_0\restriction_{\mathfrak{g}} &=0
\end{align*}
On the contrary, the spectrum of the operator $a_i$ is given by the disk $\bar{D}:=\{z\in\mathbb{C}: |z|\leq 1\}$, for each $i=1,\ldots,n$.
\end{rem}
\begin{thm}
\label{thm:faithfulQuotient}
Let $\pi_{F}:=\pi_{F}^{z}$ with $z =1$, the single Fock representation defined in \eqref{fockRep}. Then $\pi_F$ is a faithful representation for the quotient algebra $\displaystyle \mathcal{O}_A/\mathcal{I}$, where $\mathcal{I}=<a_0-a_0^{\ast}a_0>$.
\end{thm}
\begin{proof}
Let $\mathfrak{A}:=C^*(\tilde{\beta}_0,\tilde{\beta}_1,\ldots,\tilde{\beta}_n)$, the subalgebra of $\mathcal{B}(\widetilde{\ch})$ generated by the operators $\tilde{\beta}_i$, $i\in\S$, given in \eqref{betatilde}. \\
Let $\mathfrak{B}:=C^*(\tilde{\beta}_0\tilde{\beta}_0^{\ast},\tilde{\beta}_1,\ldots,\tilde{\beta}_n)$ the subalgebra of $\mathfrak{A}$, generated by the operators $\tilde{\beta}_0\tilde{\beta}_0^{\ast},\tilde{\beta}_1,\ldots,\tilde{\beta}_n$.\\
By Remark \ref{rem:spectrum}, it follows that $\tilde{\beta}_0\tilde{\beta}_0^{\ast} \Om_{z}=\Omega_{z}$ for each $z \in \mathbb{T}$. In addition, the restriction of the Fock representations $\pi^{z}_F\restriction_{\mathfrak{B}}$ is isomorphic to the  restriction of the Fock representations $\pi^{z=1}_F\restriction_{\mathfrak{B}}$ on the same $C^*$-algebra $\mathfrak{B}$, for each $z\in\mathbb{T}$. Therefore $\pi^{z=1}_F$ is a faithful representation for the $C^*$-algebra $\mathfrak{B}$. It remains to prove that the $C^*$-algebra $\mathfrak{B}$ is isomorphic to $\displaystyle \mathcal{O}_A/\mathcal{I}$.\\
First, we observe that the operators $\tilde{\beta}_0\tilde{\beta}_0^{\ast},\tilde{\beta}_1,\ldots,\tilde{\beta}_n$ satisfy the Cuntz-Krieger conditions $CK1)$ and $CK2)$ for the given matrix $A$. Therefore, it is well defined an application $\Psi:\mathcal{O}_A\rightarrow \mathfrak{B}$, such that
\begin{align*}
    \Psi(a_0)&=\tilde{\beta}_0\tilde{\beta}_0^{\ast} & &\text{and} & \Psi(a_i)&=\tilde{\beta}_i \quad i=1,\ldots,n\,.
\end{align*}
Now we prove that for each $f\in\mathcal{I}$, one has $\Psi(f)=0$. Indeed
\begin{equation*}
    \begin{split}
       \Psi(a_0-a_0^{\ast}a_0) &= \tilde{\beta}_0\tilde{\beta}_0^{\ast}- (\tilde{\beta}_0\tilde{\beta}_0^{\ast})^{\ast}\tilde{\beta}_0\tilde{\beta}_0^{\ast}\\
                               &= \tilde{\beta}_0\tilde{\beta}_0^{\ast}-\tilde{\beta}_0\tilde{\beta}_0^{\ast}=0
    \end{split}
\end{equation*}
Then it is well defined the epimorphism $\displaystyle \psi:\mathcal{O}_A/\mathcal{I}\rightarrow \mathfrak{B}$ as follows:
\begin{align*}
    \psi([a_0])&=\tilde{\beta}_0\tilde{\beta}_0^{\ast} & \text{and}& &\psi([a_i])&=\tilde{\beta}_i \quad i=1,\ldots,n\,,
\end{align*}
where $[a_i]$ denotes the equivalence class for the operator $a_i$, for each $i\in\S$. Finally, we prove that $\psi$ is a bijection. To this end, let us define the application $\eta:\mathfrak{B}\rightarrow \mathcal{O}_A/\mathcal{I}$ as follows:
\begin{align*}
\eta(\tilde{\beta}_0\tilde{\beta}_0^{\ast})&=[a_0] & &\text{and} &  \eta(\tilde{\beta}_i)&=[a_i] \quad \,i=1,\ldots,n\,.
\end{align*}
We notice that $\mathfrak{B}\subset \mathcal{O}_A$ since $\tilde{\pi}$ is a faithful representation of $\mathcal{O}_A$ by Proposition \ref{prop:faithful} and $\mathfrak{B}\subset \mathfrak{A}\equiv\tilde{\pi}(\mathcal{O}_A)$. Therefore $\eta=p\circ i$, where $p$ is the projection from the $C^*$-algebra $\mathcal{O}_A$  to the quotient $\displaystyle \mathcal{O}_A/\mathcal{I}$ and $i$ denotes the immersion from $\mathfrak{B}$ to $\mathcal{O}_A$. As a consequence $\eta$ is a well defined continuous map from $\mathfrak{B}$ to $\mathcal{O}_A/\mathcal{I}$. To conclude the proof we prove that
\begin{align*}
\psi\circ\eta &=I_{\mathfrak{B}} &  &\text{and} & \eta\circ\psi &=I_{\mathcal{O}_A/\mathcal{I}}\,,
\end{align*}
where $I_{\mathfrak{B}}$ and $I_{\mathcal{O}_A/\mathcal{I}}$ denote respectively the identity operator on the algebras $\mathfrak{B}$ and $\mathcal{O}_A/\mathcal{I}$. More in detail,
    \begin{equation*}
     \begin{split}
        (\psi\circ\eta)(\tilde{\beta}_0\tilde{\beta}_0^{\ast}) &=\psi([a_0])=\tilde{\beta}_0\tilde{\beta}_0^{\ast}\\
        (\psi\circ\eta)(\tilde{\beta}_i) &=\psi([a_i])=\tilde{\beta}_i \quad i=1,\ldots,n\,,
      \end{split}
      \end{equation*}
          and
   \begin{equation*}
     \begin{split}
        (\eta\circ\psi)([a_0]) &= \eta(\tilde{\beta}_0\tilde{\beta}_0^{\ast}) =[a_0]\\
        (\eta\circ\psi)([a_i]) &= \eta(\tilde{\beta}_i) =[a_i]\quad i=1,\ldots,n\,.
      \end{split}
      \end{equation*}
\end{proof}
We conclude this section by observing that the quotient algebra $\displaystyle \mathcal{O}_A/\mathcal{I}$ fails to be a Cuntz-Krieger algebra, since the ideal $\mathcal{I}=<a_0-a_0^{\ast}a_0>$ is not gauge-invariant. 

\section{A MASA for the weakly-monotone $C^*$-algebra}
\label{sec:MASA}
In this section, we show that the diagonal subalgebra is a maximal abelian subalgebra, \emph{MASA} from now on, for the weakly monotone $C^*$-algebra $C^*(A_1,\ldots,A_n)$. We start by analyzing, in Section \ref{subsec:MASA2}, the case of $n=2$ generators, and successively, we extend the results to the $C^*$-algebra generated by $n$ annihilation operators.

\subsection{A MASA for the weakly-monotone $C^*$-algebra $C^*(A_1,A_2)$}
\label{subsec:MASA2}
This section aims to give a MASA for the $C^*$-algebra $C^*(A_1,A_2)$ generated by two annihilation operators acting on the weakly monotone Fock space $\mathfrak{F}_{WM}(\ch)$, with $\dim{\ch}=n=2$.\\
In this case, the basis $\mathbf{B}_2$ for $\gf_{WM}(\ch)$ is given by
\begin{equation*}
\mathbf{B}_2=\{\Om,\, e_2^{k}\otimes e_1^{l}: k, l\in \mathbb{N}, k+l>0\}\,.
\end{equation*}
Denote by $\S_2:=\{0,1,2\}$. Let $\alpha=(\alpha_1,\alpha_2,\ldots,\alpha_k)$, $\alpha_j\in\Sigma_2$, a multi-index and denote by $A_{\alpha}:=A_{\alpha_1}A_{\alpha_2}\cdots A_{\alpha_k}$. In addition, we consider the empty set $\emptyset$ a multi-index with the convention $A_{\emptyset}:=I$. Let us define
\begin{equation}
\label{algD2}
    \mathfrak{D}_2:=\{C^*(I,A^{\ast}_{\alpha}A_{\alpha}):\, \alpha=(\alpha_1,\alpha_2,\ldots,\alpha_k)\,, \a_i\in\S_2\forall i=1,\ldots,k\}\,.
\end{equation}
We observe that $\mathfrak{D}_2$ is a self-adjoint sub-algebra of $C^*(A_1,A_2)$ and, in the next lines, we prove that $\mathfrak{D}_2$ is a MASA.
\begin{rem}
\label{rem:linfinity}
Let us consider the Hilbert space $\ell_2(\mathbb{N})$ with orthonormal basis $\mathfrak{B}:=\{e_i:i\in \mathbb{N}\}$. For every $i\in \mathbb{N}$ denote by $P_i$ the orthogonal projection onto the linear span generated by the vector $e_i$. Then the diagonal MASA corresponding to the orthonormal basis $\mathfrak{B}$ is given by
$$
\mathcal{N}:=\left\{\sum_{i\in \mathbb{N}} x_iP_i : x=\{x_i\}_{i\in \mathbb{N}} \in \ell_{\infty}(\mathbb{N}) \right\}\,.
$$
In particular, $\ell_{\infty}(\mathbb{N})$ is the canonical diagonal MASA in $\mathcal{B}(\ell_2(\mathbb{N}))$ \cite{Takesaki1}. As is known, an expectation $\mathbf{E}:\mathcal{B}(\ell_2(\mathbb{N})) \rightarrow \ell_{\infty}(\mathbb{N})$ is well defined by
\begin{equation*}
\langle\mathbf{E}[T]e_i,e_j\rangle=\langle Te_i,e_j\rangle\delta_{ij}\,,\quad T\in \mathcal{B}(\ell_2(\mathbb{N}))
\end{equation*}
where $\langle\cdot,\cdot\rangle$ denotes the scalar product on $\ell_2(\mathbb{N})$ and $\delta_{ij}$ is the Kronecker symbol.
\end{rem}
To show that  $\mathfrak{D}_2$ is a maximal abelian subalgebra of $C^*(A_1, A_2)$, we start by proving two preparatory results.  For any pair $(k, l)$ with $k+l>0$, denote by $P_{k,l}$ the rank-one orthogonal projection onto $\mathbb{C}e_2^k\otimes e_1^l$.
 \begin{lem}
\label{lem:bicommutant2}
Denote by $\mathfrak{D}_2^{\prime\prime}$ the bicommutant (in the Fock representation) of the $C^*$-algebra $\mathfrak{D}_2$ defined in \eqref{algD2}. Then $\mathfrak{D}_2^{\prime\prime}$ coincides with the $C^*$-algebra of all diagonal operators.
\end{lem}
\begin{proof}
We will show that the orthogonal projections $P_{\Om}$ and $P_{k,l}$, $k+l>0$, belong to $\mathfrak{D}_2$. The statement will then follow since the linear span of the above projections is weakly dense in the algebra of all diagonal operators w.r.t. the canonical basis of the Fock space.\\
We notice that $P_{\Om}=A^{\ast}_0A_0$, and it is easy to see that for each $k+l>0$, one has
\begin{equation*}
    P_{k,l}={A^{\ast}_2}^k {A^{\ast}_1}^l {A_1}^l {A_2}^k -  {A^{\ast}_2}^{k+1} {A^{\ast}_1}^l {A_1}^l {A_2}^{k+1} - {A^{\ast}_2}^k {A^{\ast}_1}^{l+1} {A_1}^{l+1} {A_2}^k\,.
\end{equation*}
\end{proof}
We introduce the following notation concerning the elements of the basis $\mathbf{B}_2$ for the weakly monotone Fock space $\gf_{WM}(\ch)$.
\begin{align*}
    \Om &:=\varepsilon_{0,0}\\
    e_2^k\otimes e_1^l &:= \veps_{k,l} \,, k+l>0\,.
\end{align*}
\begin{lem}
\label{lem:expectation2}
Let $\mathbf{E}:\mathcal{B}(\gf_{WM}(\ch))\rightarrow \mathfrak{D}_2^{\prime\prime}$ the canonical expectation satisfying
\begin{equation*}
\langle\mathbf{E}[T]\veps_{i,j},\veps_{k,l}\rangle:=\langle T\veps_{i,j},\veps_{k,l}\rangle\delta_{ik}\delta_{jl}\,,\quad T\in\mathcal{B}(\gf_{WM}(\ch))\,.
\end{equation*}
Then $\mathbf{E}(C^*(A_1,A_2))\equiv \mathfrak{D}_2$.
\end{lem}
\begin{proof}
Let $\a=(\a_1,\a_2,\ldots,\a_v)$ and $\beta=(\b_1,\b_2,\ldots,\b_w)$ be two multi-indices, with length $|v|$ and $|w|$ respectively, where $\a_i$, $\b_j\in\S_2$. For $T:=A^{\ast}_{\b}A_{\a}$ one has
\begin{equation*}
    \langle\mathbf{E}[A^{\ast}_{\b}A_{\a}]\veps_{i,j},\veps_{k,l}\rangle=\langle A^{\ast}_{\b}A_{\a}\veps_{i,j},\veps_{k,l}\rangle\delta_{ik}\delta_{jl}\,.
\end{equation*}
Suppose that $\a_i=\b_j=0$, for each $i=1,\ldots,v$, $j=1,\ldots,w$. Then
\begin{equation*}
    \langle A^{\ast}_{\b}A_{\a}\veps_{0,0},\veps_{0,0}\rangle =1\,.
\end{equation*}
In this case the lengths $|v|$ and $|w|$ can be different. Anyway, the operator $A^{\ast}_{\b}A_{\a}\in\mathfrak{D}_2$, since $A^{\ast}_0A_0=P_{\Om}$. \\

For the other cases, we need to prove that
\begin{equation}
\label{eq:diag} \tag{1}
    \langle A^{\ast}_{\b}A_{\a}\veps_{k,l},\veps_{k,l}\rangle=\left\{
       \begin{array}{ll}
         1, & \hbox{\text{if $\alpha=\beta$};} \\
         0, & \hbox{\text{if $\alpha\neq\beta$.}}
       \end{array}
     \right.
\end{equation}
Suppose that $\a=\emptyset$. Then $\langle A^{\ast}_{\b}(\veps_{0,0}),\veps_{0,0}\rangle=0$ if $\b\neq \emptyset$. Indeed,
\begin{equation*}
    A^{\ast}_{\b}(\veps_{0,0})=\left\{
       \begin{array}{ll}
         \veps_{\bar{k},\bar{l}}\,,\bar{k}+\bar{l}>0, & \hbox{\text{if $\b_w\geq\b_{w-1}\geq\ldots\geq\b_1$};} \\
         0, & \hbox{\text{otherwise.}}
       \end{array}
     \right.
\end{equation*}
If $\a\neq \emptyset$, one has $A^{\ast}_{\b}A_{\a}(\veps_{0,0})=0$. \\
Now we take into account $A^{\ast}_{\b}A_{\a}(\veps_{k,l})$, $k+l>0$. If $\a=\emptyset=\b$, then $A^{\ast}_{0}A_{0}(\veps_{k,l})=P_{\Om}(\veps_{k,l})=0$.\\
If $\a=\emptyset$ and $\b\neq\emptyset$, with $\b_w\geq \b_{w-1}\geq \ldots\geq\b_1$, then $A^{\ast}_{\b}(\veps_{k,l})=\veps_{\bar{k},\bar{l}}$, with $\bar{k}>k$ or/and $\bar{l}>l$. As a consequence, $\langle A^{\ast}_{\b}(\veps_{k,l}),\veps_{k,l}\rangle=0$. \\
We are reduced to proving \eqref{eq:diag} for $\a_i=1$ for each $i=1,\ldots,l^{\prime}$ and $\a_i=2$, for each $i=l^{\prime}+1,\ldots,v$, where $l^{\prime}\leq l$ and $k^{\prime}:=v-l^{\prime}\leq k$. Then,
\begin{equation*}
 A_{\a}(\veps_{k,l})=\left\{
       \begin{array}{ll}
       \veps_{0,0} & \hbox{\text{if $l^{\prime}=l$ and $k^{\prime}=k$};} \\
       \veps_{k-k^{\prime},l} & \hbox{\text{if $l^{\prime}=0$ and $k^{\prime}< k$};} \\
       \veps_{0,l-l^{\prime}} & \hbox{\text{if $l^{\prime}<l$ and $k^{\prime}= k$};}\\
       0 & \hbox{\text{otherwise}.}\\
       \end{array}
     \right.
\end{equation*}
If $\b=\a$, then $A^{\ast}_{\b}A_{\a}$ coincides with the identity in each of the previous cases and then $\langle A^{\ast}_{\b}A_{\a}(\veps_{k,l}),\veps_{k,l}\rangle=1$. Suppose now that $\b_j=1$ for each $j=1,\ldots,\bar{k}$ and $\b_j=2$, for each $j=\bar{k}+1,\ldots,w$, and denote by $\bar{l}=w-\bar{k}$. One has
\begin{align*}
  A^{\ast}_{\b}(\veps_{0,0}) &=  \veps_{\bar{l},\bar{k}}\,;  \\
 A^{\ast}_{\b}(\veps_{k-k^{\prime},l}) &= \left\{
       \begin{array}{ll}
       0 & \hbox{\text{if $\bar{k} \neq 0$};} \\
       \veps_{k-k^{\prime}+\bar{l},0} & \hbox{\text{if $\bar{k}=0$};}
       \end{array}
     \right.\\
A^{\ast}_{\b}(\veps_{0,l-l^{\prime}})&=\veps_{\bar{l},l-l^{\prime}+\bar{k}}\,.
\end{align*}
Therefore $\langle A^{\ast}_{\b}A_{\a}(\veps_{k,l}),\veps_{k,l}\rangle=0$ if $\bar{k}\neq k^{\prime}$ or $\bar{l}\neq l^{\prime}$.\
The computations above show that  $\mathbf{E}[A^{\ast}_{\b}A_{\a}]$  sits in $\mathfrak{D}_2$ for all multi-indices
$\a$ and $\b$. This readily implies the thesis by norm continuity of $\mathbf{E}$, as any $T$ in $C^*(A_1,A_2)$ is the norm limit
of a sequence of linear combinations of monomials of the type above.
\end{proof}

\begin{thm}
\label{thm:MASA2}
The sub-algebra $\mathfrak{D}_2$ is a MASA for the $C^*$-algebra $C^*(A_1,A_2)$.
\end{thm}
\begin{proof}
We need to prove that
\begin{equation*}
   \mathfrak{D}_2^{\prime}\cap C^*(A_1,A_2)=\mathfrak{D}_2\,,
\end{equation*}
where $\mathfrak{D}^{\prime}$ denotes the commutant of $\mathfrak{D}$ in $\mathcal{B}(\gf_{WM}(\ch))$.\\
By Lemma \ref{lem:bicommutant2}, it follows that the bicommutant $\mathfrak{D}_2^{\prime\prime}$ is a MASA of $\mathcal{B}(\gf_{WM}(\ch))$. Therefore, one has
\begin{equation}
\label{D2D1}
\mathfrak{D}_2^{\prime\prime}\equiv \mathfrak{D}_2^{\prime}\,.
\end{equation}
Fix $X\in \mathfrak{D}_2^{\prime}\cap C^*(A_1,A_2)$. By \eqref{D2D1}, it follows that $X\in \mathfrak{D}_2^{\prime\prime}\cap C^*(A_1,A_2)$. Since $X\in \mathfrak{D}_2^{\prime\prime}$, then $X=\mathbf{E}[X]$. On the other hand, $X\in C^*(A_1,A_2)$ implies $\mathbf{E}[X]\in \mathcal{D}_2$, by Lemma \ref{lem:expectation2}, and then $X\in \mathfrak{D}_2$.
\end{proof}

\subsection{A MASA for the weakly-monotone $C^*$-algebra $C^*(A_1,\ldots,A_n)$}
\label{subsec:MASAn}
In this section, we determine a MASA for the weakly monotone $C^*$-algebra generated by $n\geq 3$ annihilation operators $A_i$, $i=1,2,\ldots,n$ acting on the weakly monotone Fock space $\gf_{WM}(\ch)$ and we prove that
considerations similar to those made in the Section \ref{subsec:MASA2} can be obtained. \\

Let $\alpha=(\alpha_1,\alpha_2\ldots,\alpha_k)$, $\alpha_j\in\Sigma$, a multi-index and denote by $A_{\alpha}:=A_{\alpha_1}A_{\alpha_2}\ldots A_{\alpha_k}$. Denote by
\begin{equation}
\label{algD}
    \mathfrak{D}:=\{C^*(I,A^{\ast}_{\alpha}A_{\alpha}):\alpha=(\alpha_1,\alpha_2,\ldots,\alpha_k)\,,\a_i\in\S\, \forall i=1,\ldots,k\}\,.
\end{equation}
Following the same reasoning  as in Section \ref{subsec:MASA2}, we can prove that $\mathfrak{D}$ is a MASA for the $C^*$-algebra $C^*(A_1,A_2,\ldots,A_n)$.
In particular, following Lemma \ref{lem:bicommutant2}, let us denote  $P_{k_n, k_{n-1},\ldots,k_1}$ the orthogonal projection onto the subspace generated by $e_n^{k_n}\otimes e_{n-1}^{k_{n-1}}\otimes\cdots\otimes e_1^{k_1}$. \begin{lem}
\label{lem:bicommutant}
Denote by $\mathfrak{D}^{\prime\prime}$ the bicommutant (in the Fock representation) of the $C^*$-algebra $\mathfrak{D}$ defined in \eqref{algD}. Then $\mathfrak{D}^{\prime\prime}$ coincides with the $C^*$-algebra of all diagonal operators.
\end{lem}
\begin{proof}
We proceed analogously to the proof of Lemma \ref{lem:bicommutant2} and we prove that the projections $P_{\Om}$ and $P_{k_n, k_{n-1},\ldots,k_1}$ belongs to $\mathfrak{D}$. Indeed, $P_{\Om}=A^*_0A_0$, and  
\begin{equation*}
   P_{k_n, k_{n-1},\ldots,k_1}=A^{\ast}_{\a}A_{\a}-\sum_{h=1}^{n}A^{\ast}_{\a_h}A_{\a_h}\,,
\end{equation*}
where
\begin{align*}
    \a &:=(\underbrace{1,\ldots,1}_{k_1},\underbrace{2,\ldots,2}_{k_2},\ldots,\underbrace{n,\ldots,n}_{k_n})\\
    \a_h &:=(\underbrace{1,\ldots,1}_{k_1},\ldots,\underbrace{h,\ldots,h}_{k_h+1},\ldots,\underbrace{n,\ldots,n}_{k_n})\,,
\end{align*}
for each $h=1,\ldots,n$.
\end{proof}
We introduce the following notation concerning the elements of the basis $\mathbf{B}$ for the weakly monotone Fock space $\gf_{WM}(\ch)$.
\begin{align*}
       &\Om := \veps_{\underbrace{0,\ldots,0}_n}\\
       &e_n^{k_n}\otimes e_{n-1}^{k_{n-1}}\otimes\cdots\otimes e_1^{k_1}:=\veps_{k_n, k_{n-1},\ldots,k_1}\,,
\end{align*}
where $k_n, k_{n-1},\ldots,k_1\geq 0$, with $\displaystyle \sum_{j=1}^n k_j>0$.
\begin{lem}
\label{lem:expectation}
Let $\mathbf{E}:\mathcal{B}(\gf_{WM}(\ch))\rightarrow \mathfrak{D}^{\prime\prime}$ the canonical expectation satisfying
\begin{equation*}
\langle\mathbf{E}[T]\veps_{i_n,\ldots,i_1},\veps_{k_n,\ldots,k_1}\rangle:=\langle T\veps_{i_n,\ldots,i_1},\veps_{k_n,\ldots,k_1}\rangle\delta_{i_nk_n}\cdots\delta_{i_1k_1}\,,\quad T\in\mathcal{B}(\gf_{WM}(\ch))\,.
\end{equation*}
Then $\mathbf{E}(C^*(A_1,\ldots,A_n))\equiv \mathfrak{D}$.
\end{lem}
\begin{proof}
Let $\a=(\a_1,\a_2,\ldots,\a_v)$ and $\beta=(\b_1,\b_2,\ldots,\b_w)$ be two multi-indices, with length $|v|$ and $|w|$ respectively, where $\a_i$, $\b_j\in\S$.
For $T:=A^{\ast}_{\b}A_{\a}$ one has
\begin{equation*}
    \langle\mathbf{E}[A^{\ast}_{\b}A_{\a}]\veps_{i_n,\ldots,i_1},\veps_{k_n,\ldots,k_1}\rangle=\langle A^{\ast}_{\b}A_{\a}\veps_{i_n,\ldots,i_1},\veps_{k_n,\ldots,k_1}\rangle\delta_{i_nk_n}\cdots\delta_{i_1k_1}\,.
\end{equation*}
Suppose that $\a_i=\b_j=0$, for each $i=1,\ldots,v$, $j=1,\ldots,w$. Then
\begin{equation*}
    \langle A^{\ast}_{\b}A_{\a}\veps_{0,\ldots,0},\veps_{0,\ldots,0}\rangle =1\,.
\end{equation*}
In this case the lengths $|v|$ and $|w|$ can be different. Anyway, the operator $A^{\ast}_{\b}A_{\a}\in\mathfrak{D}$, since $A_0$ is self-adjoint and idempotent. \\
For the other cases, we need to prove that
\begin{equation*}
    \langle A^{\ast}_{\b}A_{\a}\veps_{i_n,\ldots,i_1},\veps_{k_n,\ldots,k_1}\rangle=\left\{
       \begin{array}{ll}
         1, & \hbox{\text{if $\alpha=\beta$};} \\
         0, & \hbox{\text{if $\alpha\neq\beta$.}}
       \end{array}
     \right.
\end{equation*}
Suppose that $\a=\emptyset$. Then $\langle A^{\ast}_{\b}(\veps_{0,\ldots,0}),\veps_{0,\ldots,0}\rangle=0$ if $\b\neq \emptyset$. Indeed,
\begin{equation*}
    A^{\ast}_{\b}(\veps_{0,\ldots,0})=\left\{
       \begin{array}{ll}
         \veps_{\bar{k}_n, \bar{k}_{n-1},\ldots,\bar{k}_1}\,,\sum_{j=1}^n \bar{k}_j>0, & \hbox{\text{if $\b_w\geq\b_{w-1}\geq\ldots\geq\b_1$};} \\
         0, & \hbox{\text{otherwise.}}
       \end{array}
     \right.
\end{equation*}
If $\a\neq \emptyset$, one has $A^{\ast}_{\b}A_{\a}(\veps_{0,\ldots,0})= 0$.\\
Now we take into account $A^{\ast}_{\b}A_{\a}(\veps_{k_n,\ldots,k_1})$, $\displaystyle \sum_{j=1}^n k_j>0$. If $\a=\emptyset=\b$, then $A^{\ast}_{0}A_{0}(\veps_{k_n,\ldots,k_1})=P_{\Om}(\veps_{k_n,\ldots,k_1})=0$.\\
If $\a=\emptyset$ and $\b\neq\emptyset$, with $\b_w\geq \b_{w-1}\geq \ldots\geq\b_1$, then, if $A^{\ast}_{\b}$ is not vanishing, on has $A^{\ast}_{\b}(\veps_{k_n,\ldots,k_1})=\veps_{\bar{k}_n, \bar{k}_{n-1},\ldots,\bar{k}_1}$, with $\bar{k}_i>k_i$ for at least one $i$. As a consequence, $\langle A^{\ast}_{\b}(\veps_{k_n,\ldots,k_1}),\veps_{k_n,\ldots,k_1}\rangle=0$.  \\
Suppose now that $\a_1\leq \a_2\leq\ldots\leq \a_v$. Then, $A_{\a}(\veps_{k_n,\ldots,k_1})=\veps_{0,\ldots,0}$ or $A_{\a}\veps_{k_n,\ldots,k_1}=\veps_{\bar{k}_n, \bar{k}_{n-1},\ldots,\bar{k}_1}$, with $\bar{k}_i\leq k_i$ for each $i$ and $\bar{k}_i< k_i$ for at least one $i$, or $A_{\a}\veps_{k_n,\ldots,k_1}=0$. \\
If $\b=\a$, then $A^{\ast}_{\b}A_{\a}$ coincides with the identity in each of the previous cases and then $\langle A^{\ast}_{\b}A_{\a}(\veps_{k_n,\ldots,k_1}),\veps_{k_n,\ldots,k_1}\rangle=1$. On the contrary, if there exists $j=1,\ldots,w$ s.t.  $\b_j\neq \a_j$, with $\b_w\geq\b_{w-1}\geq\ldots\geq\b_1$, then
\begin{equation*}
\begin{split}
  A^{\ast}_{\b}A_{\a}(\veps_{k_n,\ldots,k_1}) &=A^{\ast}_{\b} (\veps_{\bar{k}_n, \bar{k}_{n-1},\ldots,\bar{k}_1})\\
  &= \veps_{\tilde{k}_n, \tilde{k}_{n-1},\ldots,\tilde{k}_1} \,,
\end{split}
\end{equation*}
with $\tilde{k}_i\geq \bar{k}_i$ and $\tilde{k}_i\neq k_i$ for at least one $i$. Therefore $\langle A^{\ast}_{\b}A_{\a}(\veps_{k_n,\ldots,k_1}),\veps_{k_n,\ldots,k_1}\rangle=0$.\\
The assertion follows as in the proof of Lemma \ref{lem:expectation2}. 
\end{proof}

\begin{thm}
\label{thm:MASA}
The sub-algebra $\mathfrak{D}$ is a MASA for the $C^*$-algebra $C^*(A_1,\ldots,A_n)$.
\end{thm}
The proof follows by Lemma \ref{lem:bicommutant} and Lemma \ref{lem:expectation}, following the same reasoning as in the proof of Theorem \ref{thm:MASA2}.

\subsection{Spectrum of the MASA}
\label{subsec:spectrum}
The maximal abelian subalgebra $\D_n\subset C^*(A_1, \ldots, A_n)$ is unital, hence it is isomorphic to the $C^*$-algebra $\D_n \cong C(\DN)$ of continuous functions on a compact space $\DN$, called the Gelfand space (or the spectrum). For the description of $\DN$, we shall use some additional notation. 
Let $\M_n=(\N_0)^n$ be the collection of all sequences of $n$ nonnegative integers $\M_n:=\{(\mu_1, \ldots , \mu_n): \mu_j\in \N\cup\{0\}, j=1,\ldots , n\}$. Given $\mu\in\M_n$ we consider partial isometries 
\[
A_{\mu}:= (A_{1})^{\mu_1}\cdots (A_{n})^{\mu_n}, \quad  
A_{\mu}^*:= (A_{n}^*)^{\mu_n}\cdots (A_{1}^{*})^{\mu_1}
\]
and orthogonal projections 
\[
P_{\mu}:=A_{\mu}^*A_{\mu}, \quad P_{\mu}^0:=A_{\mu}^*P_{\Omega}A_{\mu}.
\]
Then $\D_n$ is the closure of the linear span of $\{P_{\mu}, P_{\mu}^0:\ \mu\in \M_n\}$. Comparison of these projections can be described by some partial order on the indexes. 
\begin{defin}
\label{def:indexes}
For $\mu:=(\mu_1, \ldots , \mu_n)\neq (\nu_1, \ldots , \nu_n)=\nu \in \M_n$ we set $\nu \prec \mu$ if there exists $k\in \{2, \ldots , n-1\}$ for which 
\begin{eqnarray*}
\nu_j=\mu_j&\text{for}& k+1\leq j \leq n \\
\nu_k<\mu_k && \\ 
0=\nu_j\leq \mu_j &\text{for}& 1\leq j \leq k-1.
\end{eqnarray*}
\end{defin}
This relation describes nontrivial products of two projections. 

\begin{prop}
\label{prop:projprod}
For $\mu\neq\nu\in\M_n$ assume that $P_{\mu}P_{\nu}\neq 0$. Then $P_{\mu}P_{\nu}=P_{\mu}$  if and only if   $\nu\prec \mu$.
\end{prop}
\begin{proof}
We start with the following simple observation. Let $1\leq i,j,k,m \leq n$ be so that $i<j$ and $k<m$, and define two projections
\begin{eqnarray*}
P_{\mu}&:=&(A_j^*)^{\mu_j}(A_i^*)^{\mu_i}A_i^{\mu_i}A_j^{\mu_j}, \quad \mu_i,\mu_j \geq 1\\  P_{\nu}&:=&(A_m^*)^{\nu_m}(A_k^*)^{\nu_k}A_k^{\nu_k}A_m^{\nu_m}, \quad \nu_k, \nu_m\geq 1
\end{eqnarray*}
Then for their product 
\[
P_{\mu}P_{\nu}=(A_j^*)^{\mu_j}(A_i^*)^{\mu_i}A_i^{\mu_i}\big[A_j^{\mu_j} (A_m^*)^{\nu_m}\big](A_k^*)^{\nu_k}A_k^{\nu_k}A_m^{\nu_m}
\]
to be nonzero it must be $m=j$, otherwise $A_j^{\mu_j} (A_m^*)^{\nu_m}=0$, since $A_m^*$ is creation by $e_m$, which is annihilated by $A_j$ for $j\neq m$. Moreover, assuming now $j=m$, if $\mu_j > \nu_j$ then $A_j^{\mu_j} (A_j^*)^{\nu_j}=A_j^{c_j}$ with $c_j:=\mu_j-\nu_j\geq 1$, and then $A_j^{c_j}(A_k^*)^{\nu_k}=0$, since $j=m>k$. Similarly, if $\mu_j < \nu_j$ then $A_j^{\mu_j} (A_j^*)^{\nu_j}=(A_j^*)^{d_j}$ with $d_j:=\nu_j-\mu_j\geq 1$, and consequently $A_i^{\mu_i}(A_j^*)^{d_j}=0$, since $j>i$. 

Therefore, for $P_{\mu}P_{\nu}\neq 0$ it must be $j=m$ and $\mu_j=\nu_j$, which implies that the projection $A_j^{\mu_j} (A_j^*)^{\mu_j}$ can be removed from the product, resulting in  
\[
P_{\mu}P_{\nu}=(A_j^*)^{\mu_j}(A_i^*)^{\mu_i}A_i^{\mu_i}(A_k^*)^{\nu_k}A_k^{\nu_k}A_j^{\mu_j}
\]
The same arguments as above apply to $A_i^{\mu_i}(A_k^*)^{\nu_k}$ resulting in the conclusion $i=k$ and $\mu_i=\nu_i$. This means that $\mu=\nu$ and the product has the form 
\[
P_{\mu}P_{\nu}=(A_j^*)^{\mu_j}(A_i^*)^{\mu_i}A_i^{\mu_i}A_j^{\mu_j}=P_{\mu}^2=P_{\mu}.
\]

In the general case, if $\mu=(\mu_1, \ldots , \mu_n)$ and $\nu=(\nu_1, \ldots , \nu_n)$, then it follows from the above considerations that $\mu_n\neq\nu_n$ implies $P_{\mu}P_{\nu}=0$. Hence if $\mu\neq \nu$ and $P_{\mu}P_{\nu}\neq 0$, then there exists the least $2\leq k\leq n-1$ such that $\mu_j=\nu_j$ for $j=k+1, \ldots n$, or $k=n$. In such case $\mu_k\neq\nu_k$ and we can write 
\begin{eqnarray*}
P_{\mu}&=&(A_{n}^*)^{\mu_n}\dots (A_{1}^{*})^{\mu_1} A_{1}^{\mu_1}\dots A_{n}^{\mu_n} =B_{\mu}^*(A_k^*)^{\mu_k}C_{\mu}^*C_{\mu}A_k^{\mu_k}B_{\mu} \, \\ 
P_{\nu}&=&(A_{n}^*)^{\nu_n}\dots (A_{1}^{*})^{\nu_1} A_{1}^{\nu_1}\dots A_{n}^{\nu_n} =B_{\mu}^*(A_k^*)^{\nu_k}D_{\nu}^*D_{\nu}A_k^{\nu_k}B_{\mu}\, \\
B_{\mu}&=&A_{k+1}^{\mu_{k+1}}\dots A_{n}^{\mu_n} \,\\
C_{\mu}&=& A_{1}^{\mu_{1}}\dots A_{k-1}^{\mu_{k-1}}, \\  
D_{\nu}&=& A_{1}^{\nu_{1}}\dots A_{k-1}^{\nu_{k-1}} \\
P_{\mu}P_{\nu}&=&B_{\mu}^*(A_k^*)^{\mu_k}C_{\mu}^*C_{\mu} A_k^{\mu_k} (A_k^*)^{\nu_k}D_{\nu}^*D_{\nu}A_k^{\nu_k}B. 
\end{eqnarray*}

There are two cases: 
\begin{enumerate}
    \item $\mu_k<\nu_k$ and then $A_k^{\mu_k} (A_k^*)^{\nu_k}=(A_k^*)^{d_k}$ with $d_k=\nu_k-\mu_k\geq 1$; in such case $C_{\mu}(A_k^*)^{d_k}=0$ except $\mu_1=\dots = \mu_{k-1}=0$, i.e. when $C_{\mu}=1$. This would imply $P_{\mu}P_{\nu}=P_{\nu}$, since then $(A_k^*)^{\mu_k}A_k^{\mu_k} (A_k^*)^{\nu_k}=(A_k^*)^{\nu_k} $ .
    \item $\mu_k >\nu_k$ and then $A_k^{\mu_k} (A_k^*)^{\nu_k}=A_k^{c_k}$ with $c_k=\mu_k-\nu_k\geq 1$; in such case $A_k^{c_k}D_{\nu}^*=0$ except $\nu_1=\dots = \nu_{k-1}=0$, i.e. when $D_{\nu}=1$.  This would imply $P_{\mu}P_{\nu}=P_{\mu}$, since then 
    $A_k^{\mu_k}(A_k^*)^{\nu_k}A_k^{\nu_k}=A_k^{\mu_k}$.
\end{enumerate}
Hence both cases are equivalent by the exchange $\mu \leftrightarrow \nu$ and conclude the proof.
\end{proof}

For each $\mu\in\M_n$ the projection $P_{\mu}$ (as well as $P_{\mu}^0$) define a multiplicative functional 
$\vp_{\mu}\in \DN$ 
\begin{equation}\label{functionals}
\vp_{\mu}(P_{\nu}^0)=\vp_{\mu}(P_{\nu}):= 
\begin{cases}
1 & \text{if} \quad \nu \preceq \mu \\ 
0 & \text{otherwise} 
\end{cases}
\end{equation}
with the two exceptional cases:
\begin{eqnarray*}
P_{\Omega} &\text{defines}& \vp_0(P_{\Omega})=1, \vp_0(P_{\nu})=0 \quad \text{otherwise} \\ 
P_1=1 &\text{defines}& \vp_1(P_{\nu})=1 \quad \text{for all} \quad \nu\in\M_n, 
\end{eqnarray*}
where $P_1=1$ is the identity operator. These functionals form a dense discrete subset of $\DN^0\subset \DN$ (for the topology of pointwise convergence). The definition \eqref{functionals} is a consequence of the following observation: if $\vp$ is a multiplicative functional such that $\vp(P_{\mu})=1$ for some $\mu$, then for $\nu \prec \mu$ it is $P_{\mu}P_{\nu}=P_{\mu}$, hence it must be
\[
1=\vp(P_{\mu})=\vp(P_{\mu}P_{\nu})=\vp(P_{\mu})\vp(P_{\nu})=\vp(P_{\nu}).
\]

Now we shall identify this subset $\DN^0$ with a subset $\DD_n^0$ of the $n$-dimensional unit cube $\DN^0\cong \DD_n^0\subset [0,1]^n\subset \R^n$. For this purpose, given $\mu=(\mu_1, \ldots , \mu_n)\in \M_n$ and $k\in \{1,\ldots , n\}$, we define 
\[
r_k=r_k(\mu):= \begin{cases}
\mu_k+\dots + \mu_n & \text{if} \quad \mu_k\neq 0 \\ 
0 & \text{if} \quad \mu_k=0
\end{cases}
\]
Then, for a number $c\in (0,1)$ we identify 
\[
\DN^0\ni \vp_{\mu} \longleftrightarrow x(\mu):=(x_1,\ldots , x_n)\in \DD_n^0 \subset [0,1)^n
\]
where $x_k=x_k(\mu):=1-c^{r_k(\mu)}$, for $k=1, \ldots , n$. In particular, if $\mu_1\cdot \ldots \cdot \mu_n\neq 0$ (i.e. all coordinates are nonzero), then 
\[
\vp_{\mu} \longleftrightarrow  (1-c^{\mu_1+\dots +\mu_n}, 1-c^{\mu_2+\dots +\mu_n}, \ldots , 1-c^{\mu_n}),
\]
and if $\mu_1=\ldots = \mu_n=0$ then $\vp_{\mu} \longleftrightarrow  (0,\ldots , 0)$. To describe the closure $\overline{\DD_n^0}= \DD_n\subset [0,1]^n$ we use the notation 
\[
\kappa(\mu):=\min\{j:x_j\neq 0\ \text{and}\ x_j\neq 1 \} \quad \text{for}\quad 0\leq x_j\leq 1.
\]
If $k=\kappa(\mu)$ then we consider the sequence 
\[
\mu(p):=(\ve_1,\ldots , \ve_{k-1}, \mu_k(p), \ldots , \mu_n), \quad p=1,2,\dots 
\]
in which $\ve_1,\ldots, \ve_{k-1}\in \{0,\ 1\}$ and all the entries, except the one on $k$-th place, are constant. Letting $\mu_k(p)\to \infty$ as $p\to \infty$ we get the limit 
\[
\lim_{p\to \infty} \vp_{\mu(p)} = \vp_{\mu'}, \quad \text{with}\quad \mu':=(\ve_1,\ldots , \ve_{k-1}, 1,\mu_{k+1}, \ldots ,\mu_n),
\]
which corresponds to the boundary point 
\[
x(\mu')=(\ve_1,\ldots , \ve_{k-1}, 1, 1-c^{r_{k+1}(\mu)}, \ldots, 1-c^{r_{n}(\mu)}) \in \partial \DD_n.
\]
It follows that the accumulation points $\partial \DD_n$ of the spectrum correspond to sequences $x=(x_1, \ldots , x_n)\in [0,1]^n$ for which there exists $1\leq k \leq n$ such that $x_k=1$ and if $x_j\notin \{0,\ 1\}$, then $x_j=1-c^{r_j(\mu)}$ whenever $x=x(\mu)$. In particular, all the vertex points $x\in [0, 1]^n$ with $x_j\in \{0,\ 1\}$, are the accumulation points.  

This way we have obtained the following description of the spectrum $\DN$ of the maximal abelian subalgebra $\D_n$. 
\begin{thm}
\label{thm:spectrum}
The spectrum $\DN$ of the maximal abelian subalgebra $\D_n$ consists of the discrete part $\DN^0$ and the boundary part $\partial \DN$, which can be identified with the following subsets of the $n$-dimensional unit cube $[0,1]^n$
\[
\DN^0\ni \vp_{\mu}  \longleftrightarrow  x(\mu)=(x_1,\ldots , x_n)\in \DD_n^0,\ x_j=1-c^{r_j(\mu)}
\]
and the accumulation points on the boundary \\ 
\begin{eqnarray*}
\partial \DN \ni \vp_{\mu} &\longleftrightarrow &x(\mu)=(x_1,\ldots , x_n)\in [0,1]^n,\\ 
&\text{exists} &  1\leq k \leq n,\ x_k=1, \\ 
&\text{and}&  (x_j\notin \{0,\ 1\} \Rightarrow x_j=1-c^{r_j(\mu)}) 
\end{eqnarray*}
\end{thm}
For a better understanding of the above description of the spectrum
we present first the picture for the case $n=2$ and $c=\frac{1}{2}$. The discrete part consists of the points 
\vspace{.5cm}
\begin{equation*}
\begin{matrix} 
(0,1)& \dots &\dots &\dots &\dots &\dots &\dots  &(1,1) \\
\uparrow&\dots &\dots &\dots &\dots &\dots &\dots & \uparrow \\
\left(0, \frac{2^{k}-1}{2^k}\right),& \dots & \dots & \left(\frac{2^{k+1}-1}{2^{k+1}}, \frac{2^{k}-1}{2^k}\right), &\dots,& \left(\frac{2^{k+m}-1}{2^{k+m}}, \frac{2^{k}-1}{2^k}\right)  &\stackrel{m}{\rightarrow}&  \left(1,  \frac{2^{k}-1}{2^k} \right) \\ 
\vdots&\vdots&\vdots&\vdots&\vdots&\vdots&\vdots& \uparrow \\
\left(0, \frac{3}{4}\right),& \dots & \dots &\left(\frac{7}{8},  \frac{3}{4}\right), &\dots,& \left(\frac{2^{m}-1}{2^m}, \frac{3}{4}\right)  &\rightarrow&  \left(1,  \frac{3}{4} \right) \\ 
\left(0, \frac{1}{2}\right),&\dots & \left(\frac{3}{4}, \frac{1}{2}\right), &\left(\frac{7}{8}, \frac{1}{2}\right), &\dots,& \left(\frac{2^{m}-1}{2^m},\frac{1}{2}\right)  &\rightarrow&  \left(1, \frac{1}{2} \right) \\
(0,0),&\left(\frac{1}{2},0\right),& \left(\frac{3}{4},0\right),& \left(\frac{7}{8},0\right),&  \dots,& \left(\frac{2^{m}-1}{2^m},0\right), &\rightarrow& \left(1, 0\right) 
\end{matrix}
\end{equation*}

\vspace{.5cm}

In Figure \ref{fig:spectrum2} we present the identification of the spectrum $\DD(2)$ with $\DD_2 \subset [0,1]^2$. The red dots represent the corresponding multiplicative functionals and the blue arrows indicate their possible convergences. Observe that the red dots are on the $x_2$-axis (the axis with a higher number) converging to $(0,1)$, and the same happens on $x_1$ axis and on the interval from $(1,0)$ to $(1,1)$. Otherwise, they are always converging to the right, to the points on this vertical interval. Observe, that there are no red dots on the top side of the square. This situation is repeated in the general case, whenever $1\leq x_1<x_2\leq n$.  

\begin{figure}[H]
   \centering
   \begin{tikzpicture}
   \begin{axis}[
    xmin=-0.05, xmax=1.05, ymin=-0.15, ymax=1.05, 
    axis x line=middle, axis y line=middle, 
    enlargelimits=false,
    minor tick num=1,
    axis line style={shorten >=-20pt, shorten <=-20pt},
    xticklabel=\empty,
    xlabel style={
            anchor=west,
            at={(ticklabel* cs:1.0)},
            xshift=20pt
        },
        xlabel=$x_1$,
        ylabel style={
            anchor=south,
            at={(ticklabel* cs:1.0)},
            yshift=20pt
        },
        yticklabel=\empty,
        ylabel=$x_2$
    ]  
    \addplot [
   scatter,
   only marks,
   point meta=explicit symbolic,
   scatter/classes={
            a={mark=*,red}            
        },
    ] table [meta=label] {
   x  y  label
   0  0      a
   0  0.5    a
   0  0.75   a
   0  0.875 a
   0   1     a
   
   0.5   0    a
   0.75  0    a
   0.875 0    a 
   1     0    a
   
   1  0.5    a
   1  0.75   a
   1  0.875  a
   1   1     a
   
   0.75  0.5  a
   0.875 0.5  a
   0.9375 0.5 a
   
   0.875    0.75  a
   0.9375   0.75  a
   0.96875  0.75  a
   
   0.9375     0.875  a
   0.96875    0.875  a
   0.984375    0.875  a
   };
   
   \addplot [
   scatter,
   blue,
   solid,
   point meta=explicit symbolic,
   scatter/classes={
            a={mark=*,red},
            b={mark=triangle*,red},
            c={mark=o,draw=red},
            d={loosely dotted,mark=+}
        },
    ] table [meta=label] {
   x  y    label  
   0   0     d
   1   0     d
   
   0   1     d
   1   1     d
   
   1   1     d
   1   0     d
   
   1   0     d
   0   0     d
   
   0   0     d
   0   1     d  
   };
   
   \addplot [
   scatter,
   yellow,
   solid,
   thin,
   point meta=explicit symbolic,
   scatter/classes={
            a={mark=*,red},
            d={loosely dotted,mark=+}
        },
    ] table [meta=label] {
   x  y    label  
   0   0.5     d
   1   0.5     d
   
   0   0.75     d
   1   0.75     d
   
   0   0.875     d
   1   0.875     d
   
   0.5   0       d
   1     1       d
   
   0.75   0      d
   1     1       d
   
   0.875  0      d
   1      1      d
   };
   
   \addplot [
   scatter,
   yellow,
   solid,
   very thick,
   point meta=explicit symbolic,
   scatter/classes={
            a={mark=*,red},
            d={loosely dotted,mark=+}
        },
    ] table [meta=label] {
   x  y    label  
     
   0.5   0       d
   1     1       d
   
   0.75   0      d
   1     1       d
   
   0.875  0      d
   1      1      d
   };
   
   \draw[-{Classical TikZ Rightarrow[length=2mm]}, color=blue]  (0,0) --  (0,0.625); 
   \draw[-{Classical TikZ Rightarrow[length=2mm]}, color=blue]  (0,0.625) --  (0,1);
  
    \draw[-{Classical TikZ Rightarrow[length=2mm]}, color=blue]  (0.5,0) --  (0.625,0);
   \draw[-{Classical TikZ Rightarrow[length=2mm]}, color=blue]  (0.625,0) --  (0.875,0);
   
   \draw[-{Classical TikZ Rightarrow[length=2mm]}, color=blue]  (1,0) --  (1,0.625);
   \draw[-{Classical TikZ Rightarrow[length=2mm]}, color=blue]  (1,0.625) --  (1,0.875);
   
   \draw[-{Classical TikZ Rightarrow[length=2mm]}, color=blue]  (0.75,0.5) --  (0.875,0.5);
   \draw[-{Classical TikZ Rightarrow[length=2mm]}, color=blue]  (0.875,0.5) --  (0.9375,0.5);
   
   \draw[-{Classical TikZ Rightarrow[length=2mm]}, color=blue]  (0.875,0.75) --  (0.9375,0.75);
   \draw[-{Classical TikZ Rightarrow[length=2mm]}, color=blue]  (0.875,0.75) --  (0.96875,0.75);
  
  \draw[-{Classical TikZ Rightarrow[length=2mm]}, color=blue]  (0.875,0.875) --  (0.96875,0.875);

  \node [below left] at (0,0) {\tiny $0$};
   \node [below] at (0.5,0) {\tiny $\frac{1}{2}$};
   \node [below] at (0.75,0) {\tiny $\frac{3}{4}$};
   \node [below] at (0.875,0) {\tiny $\frac{7}{8}$};
   \node [below] at (1,0) {\tiny $1$};

   \node [left] at (0,0.5) {\tiny $\frac{1}{2}$};
   \node [left] at (0,0.75) {\tiny $\frac{3}{4}$};
   \node [left] at (0,0.875) {\tiny $\frac{7}{8}$};
   \node [left] at (0,1) {\tiny $1$};
       
   \end{axis}
   \end{tikzpicture}
      \caption{The identification of the spectrum $\Delta(2)$ with $\Delta_2\subset[0,1]^2 $.}
      \label{fig:spectrum2}
   \end{figure}
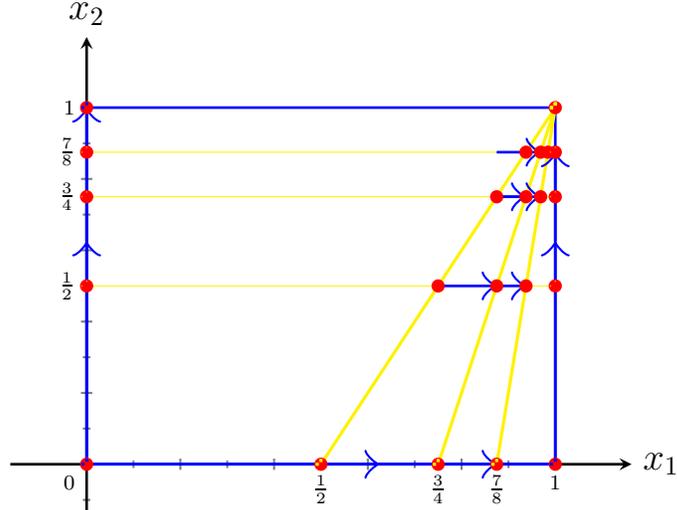

In Figure \ref{fig:spectrum3} we present the case $n=3$, with ordered axes \\ $1\leq x_1<x_2<x_3\leq n$, in which each face, except the top one, looks the same way as in the case $n=2$.

   \begin{figure}[H]
   \centering
   \begin{tikzpicture}
    \begin{axis}[
    xmin=-0.10, xmax=1.30, ymin=-0.20, ymax=1.05, zmin=-0.15, zmax=1.10,
    axis x line=middle, axis y line=middle, axis z line=middle,
    enlargelimits=false,
    minor tick num=1,
    grid=both,
    axis line style={shorten >=-20pt, shorten <=-20pt},
    xticklabel=\empty,
    xlabel style={
            anchor=west,
            at={(ticklabel* cs:1.0)},
            xshift=20pt
        },
        xlabel=$x_1$,
        yticklabel=\empty,
        ylabel style={
            anchor=west,
            at={(ticklabel* cs:1.0)},
            yshift=20pt
        },
        ylabel=$x_2$,   
        zticklabel=\empty,  
        z label style={anchor=south east},
        zlabel=$x_3$
    ]
    \addplot3 [
   scatter,
   only marks,
   point meta=explicit symbolic,
   scatter/classes={
            a={mark=*,red}
        },
    ] table [meta=label] {
   x    y    z    label
   0    0    0      a
   0    0    0.5    a
   0    0    0.75   a
   0    0    0.875  a
   0    0     1     a
   
   0  0.5    0     a
   0  0.75   0     a
   0  0.875  0     a
   0   0     1     a
   
   0.5   0   0     a
   0.75  0   0     a
   0.875 0   0     a
   
   1     0.5     0 a
   1     0.75    0 a
   1     0.875   0 a
   
   0.5        1     0 a
   0.75       1     0 a
   0.875      1     0 a
   1          1     0 a
   
   0          1     0.5   a
   0          1     0.75  a
   0          1     0.875 a
   0          1     0     a
   
   1          0       0    a
   1          0      0.5   a
   1          0      0.75  a
   1          0      0.875 a
   1          0      1     a
   
   1          1      0.5   a
   1          1      0.75  a
   1          1      0.875 a
   1          1      1     a
   
   0          0.75     0.5   a 
   0          0.875    0.5   a
   0          0.9375   0.5   a
   
   0          0.875     0.75   a 
   0          0.9375    0.75   a
   0          0.96875   0.75   a
   
   0          0.9375    0.875   a 
   0          0.96875   0.875   a
   0          0.984375  0.875   a
   0          1         1       a
   
   0.75        1         0.5    a
   0.875       1         0.5    a
   0.9375      1         0.5    a
   
   0.75        1         0.75    a
   0.875       1         0.75    a
   0.9375      1         0.75    a
   
   0.75        1         0.875    a
   0.875       1         0.875    a
   0.9375      1         0.875    a
   
    0.75           0.5         0      a
    0.875          0.5         0      a
    0.9375         0.5         0      a
                         
    0.875           0.75        0     a
    0.9375          0.75        0     a
    0.96875         0.75        0     a
                         
   0.9375            0.875       0     a
   0.96875           0.875       0     a
   0.984375          0.875       0     a
   
    0.75                 0       0.5   a
    0.875                0       0.5   a
    0.9375               0       0.5   a
    
    0.875               0         0.75   a
    0.9375              0         0.75   a
    0.96875             0         0.75   a
    
    0.9375               0        0.875    a
    0.96875              0        0.875    a
    0.984375             0        0.875    a
    
    1                 0.75       0.5   a
    1                0.875       0.5   a
    1               0.9375       0.5   a
    
   1               0.875         0.75   a
   1              0.9375         0.75   a
   1             0.96875         0.75   a
    
   1               0.9375        0.875    a
   1               0.96875       0.875    a
   1               0.984375      0.875    a                           
    
   };
   \addplot3 [
   scatter,
   yellow,
   solid,
   thin,
   point meta=explicit symbolic,
   scatter/classes={
            a={mark=*,red},
            b={mark=triangle*,red},
            c={mark=o,draw=red},
            d={loosely dotted,mark=+}
        },
    ] table [meta=label] {
   x      y      z    label  
   0    0      0.5      d
   0   1       0.5      d
   
   0    0      0.75      d
   0    1       0.75     d
   
   0    0      0.875     d
   0    1       0.875     d
   
   0     1       0.5      d
   1     1       0.5      d
      
   0     1       0.75      d
   1     1       0.75      d
   
   0     1       0.875      d
   1     1       0.875      d
   
   1     1       0.5       d
   1     0       0.5       d
   0     0       0.5       d
   
   1     1       0.75       d
   1     0       0.75       d
   0     0       0.75       d
   
   1     1       0.875       d
   1     0       0.875       d
   0     0       0.875       d
   
   0     0.5       0         d
   1     0.5       0         d
   
   0     0.75       0         d
   1     0.75       0         d
   
   0     0.875       0        d
   1     0.875       0        d
   };
   
   \addplot3 [
   scatter,
   yellow,
   solid,
   very thick,
   point meta=explicit symbolic,
   scatter/classes={
            a={mark=*,red},
            b={mark=triangle*,red},
            c={mark=o,draw=red},
            d={loosely dotted,mark=+}
        },
    ] table [meta=label] {
   x      y      z    label  
   0    0.5      0      d
   0     1       1      d
   
   0    0.75     0      d
   0     1       1      d
   
   0    0.875    0      d
   0     1       1      d

   0.5    0      0      d
   1      0      1      d
   
   0.75    0     0      d
   1       0     1      d
   
   0.875    0   0      d
   1        0   1      d

   1     0.5      0      d
   1      1       1      d
   
   1    0.75     0      d
   1       1     1      d
   
   1    0.875    0      d
   1        1    1      d

   0.5     0     0      d
   1       1     0      d
   
   0.75     0     0      d
   1       1     0      d
   
   0.875     0     0     d
   1         1     0      d

   0.5       1      0     d   
   1         1      1     d
   
   0.75       1      0     d   
   1          1      1     d
   
   0.875       1      0     d   
   1           1      1     d
   };
   
   \addplot3 [
   scatter,
   blue,
   solid,
   point meta=explicit symbolic,
   scatter/classes={
            a={mark=*,red},
            b={mark=triangle*,red},
            c={mark=o,draw=red},
            d={loosely dotted,mark=+}
        },
    ] table [meta=label] {
   x      y    z   label  
   0      0    0    d
   0      0    1    d
   0      1    1    d
   1      1    1    d
   1      1    0    d
   1      0    0    d
   
   0      0    0    d
   0      1    0    d
   0      1    1    d
   
   0      1    0    d
   1      1    0    d
   1      0    0    d
   0      0    0    d
   
   1      0    0    d
   1      0    1    d
   1      1    1    d
   
   1      0    1    d
   0      0    1    d 
   };
   
   \draw[-{Classical TikZ Rightarrow[length=2mm]}, color=blue]  (0,0,0) --  (0.70,0,0); 
   \draw[-{Classical TikZ Rightarrow[length=2mm]}, color=blue]  (0,0,0) --  (0,0,0.70);
   \draw[-{Classical TikZ Rightarrow[length=2mm]}, color=blue]  (0,1,0) --  (0,1,0.7); 
   \draw[-{Classical TikZ Rightarrow[length=2mm]}, color=blue]  (1,0,0) --  (1,0,0.7);
   \draw[-{Classical TikZ Rightarrow[length=2mm]}, color=blue]  (1,0,0) --  (1,0.7,0);
   \draw[-{Classical TikZ Rightarrow[length=2mm]}, color=blue]  (0,0,0) --  (0,0.7,0);
   \draw[-{Classical TikZ Rightarrow[length=2mm]}, color=blue]  (0,1,0) --  (0.7,1,0);
   \draw[-{Classical TikZ Rightarrow[length=2mm]}, color=blue]  (1,1,0) --  (1,1,0.7);
   
   \node [below right] at (0,0,0) {\tiny $(0,0,0)$};
   \node [right] at (0,0,1) {\tiny $(0,0,1)$};
   \node [below right] at (0,1,1) {\tiny $(0,1,1)$};
   \node [right] at (0,1,0) {\tiny $(0,1,0)$};
   \node [right] at (1,0,1) {\tiny $(1,0,1)$};   
   \node [right] at (1,1,1) {\tiny $(1,1,1)$};
   \node [right] at (1,1,0) {\tiny $(1,1,0)$};
   \node [below left] at (1,0,0) {\tiny $(1,0,0)$}; 
      
  \end{axis}
 \end{tikzpicture}
 \caption{The identification of the spectrum $\Delta(3)$ with $\Delta_3\subset[0,1]^3$.}\label{fig:spectrum3}
\end{figure}
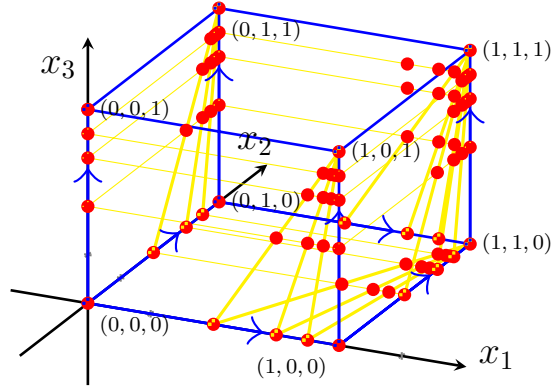

\subsection*{Acknowledgments}
The first named author is partially supported by ``INdAM-GNAMPA Project'' codice CUP$\_$E53C22001930001, Italian PNRR MUR project PE0000023-NQSTI, CUP H93C22000670006, and  Progetto ERC SEEDS UNIBA ``$C^*$-algebras and von Neumann algebras in Quantum Probability'', CUP H93C23000710001. \\
Finally, this research is part of the EU Staff Exchange project 101086394 ``Operator Algebras That One Can See''.

\end{document}